\documentclass[amstex,10pkjt,reqno]{amsart}
\usepackage{amsmath,amsfonts,amssymb,amsthm,enumerate,multicol}
\newtheorem{lemma}{Lemma}
\newtheorem{thm}{Theorem}

\numberwithin{equation}{section}
\begin{document}

\leftline{ \scriptsize \it  }
\title[]
{voronovskaja type asymptotic approximation by general Gamma type operators }
\maketitle

\begin{center}
{\bf Alok Kumar}
\vskip0.2in
Department of Computer Science\\
Dev Sanskriti Vishwavidyalaya Haridwar\\
Haridwar-249411, India\\
\vskip0.2in
 alokkpma@gmail.com\\
\vskip0.2in
 \end{center}

\begin{abstract}
In the present paper, we studied the voronovskaja type theorem for general Gamma type operators. Also, we obtain an error estimate for general Gamma type operators.
\end{abstract}
\section{\textbf{Introduction}}
In \cite{LUP}, Lupas and M$\ddot{u}$ller defined and studied the Gamma operators $G_{n}(f;x)$ as
\begin{eqnarray*}
G_{n}(f;x) = \int_{0}^{\infty}g_{n}(x,u)f\bigg(\frac{n}{u}\bigg)du,
\end{eqnarray*}
where
\begin{eqnarray*}
\displaystyle g_{n}(x,u)=\frac{x^{n+1}}{n!}e^{-xu}u^{n}, ~~~~x>0.
\end{eqnarray*}

In \cite{SM}, Mazhar gives an important modifications of the Gamma operators using the same $ g_{n}(x,u)$
\begin{eqnarray*}
 F_{n}(f;x)  &=& \int_{0}^{\infty}\int_{0}^{\infty}g_{n}(x,u)g_{n-1}(u,t)f(t)dudt\\
 &=&\frac{(2n)! x^{n+1}}{n!(n-1)!}\int_{0}^{\infty}\frac{t^{n-1}}{(x+t)^{2n+1}}f(t)dt,\,\, n>1, \,\, x>0.
\end{eqnarray*}

Recently, Karsli \cite{HK} considered a modification and obtain the rate of convergence of these operators for functions with derivatives of bounded variation.
\begin{eqnarray*}
 L_{n}(f;x)  &=& \int_{0}^{\infty}\int_{0}^{\infty}g_{n+2}(x,u)g_{n}(u,t)f(t)dudt\\
 &=&\frac{(2n+3)!x^{n+3}}{n!(n+2)!}\int_{0}^{\infty}\frac{t^{n}}{(x+t)^{2n+4}}f(t)dt, \,\, x>0.
\end{eqnarray*}
Karsli and $\ddot{O}$zarslan obtained local and global approximation results for $L_{n}(f;x)$ in \cite{HMA}. Also, Voronovskaja type asymptotic formula for $L_{n}(f;x)$ were proved in \cite{AIZ} and \cite{GK}.\\

In the year 2007, Mao \cite{MO} define the following generalized Gamma type linear and positive operators
\begin{eqnarray*}
 M_{n,k}(f;x)  &=& \int_{0}^{\infty}\int_{0}^{\infty}g_{n}(x,u)g_{n-k}(u,t)f(t)dudt\\
 &=&\frac{(2n-k+1)!x^{n+1}}{n!(n-k)!}\int_{0}^{\infty}\frac{t^{n-k}}{(x+t)^{2n-k+2}}f(t)dt, \,\, x>0,
\end{eqnarray*}
which includes the operators $F_{n}(f;x)$ for $k=1$ and $ L_{n-2}(f;x)$ for $k=2$. Some approximation properties of $M_{n,k}$ were studied in \cite{HPM} and \cite{HKL}.\\
 We can rewrite the operators $M_{n,k}(f;x)$ as
\begin{eqnarray}\label{eq1}
 M_{n,k}(f;x) = \int_{0}^{\infty}K_{n,k}(x,t)f(t)dt,
\end{eqnarray}
where
\begin{eqnarray*}
K_{n,k}(x,t)=\frac{(2n-k+1)!x^{n+1}}{n!(n-k)!}\frac{t^{n-k}}{(x+t)^{2n-k+2}}, \,\, ~~x,t\in(0,\infty).
\end{eqnarray*}

The main goal of this paper is to obtain a Voronovskaja type asymptotic formula and an error estimates for the operators (\ref{eq1}).

\section{\textbf{Auxiliary Results}}
In this section, we give some lemmas which are necessary to prove our main results.
\begin{lemma}\label{l0}\cite{HKL}
For any $m\in N_{0}$(the set of non-negative integers), $m\leq n-k$
\begin{eqnarray}
 M_{n,k}(t^{m};x)=\frac{[n-k+m]_{m}}{[n]_{m}}x^{m}.
 \end{eqnarray}
where $n, k\in N$ and $[x]_{m}=x(x-1)...(x-m+1), [x]_{0}=1, x\in R.$ \\
In particular for m = 0, 1, 2... in (2.1)  we get
\begin{enumerate}[(i)]
   \item $ M_{n,k}(1;x)= 1;$
   \item $ M_{n,k}(t;x)=\displaystyle\frac{n-k+1}{n}x;$
   \item $ M_{n,k}(t^{2};x)=\displaystyle\frac{(n-k+2)(n-k+1)}{n(n-1)}x^{2}$.
\end{enumerate}
\end{lemma}
\begin{lemma}\label{l1}\cite{HKL}
Following equalities holds:
\begin{enumerate}[(i)]
   \item $M_{n,k}((t-x);x)=\displaystyle \frac{1-k}{n}x;$
   \item $M_{n,k}((t-x)^{2};x)=\displaystyle \frac{(k^2-5k+2n+4)}{n(n-1)}x^{2};$
   \item $M_{n,k}((t-x)^{m};x)=\displaystyle O\left(n^{-[(m+1)/2]}\right)$.
\end{enumerate}
\end{lemma}

For simplicity, put $\displaystyle \beta_{n}=\frac{(2n-k+1)!}{n!(n-k)!}$.

\begin{lemma}\label{l2}
If $r^{th}$ derivative $f^{(r)}(r=0,1,2...)$ exists continuously, then we get
\begin{eqnarray*}
M_{n,k}^{(r)}(f;x)=\beta_{n}x^{n+1-r}\int_{0}^{\infty}\frac{t^{n-k+r}}{(x+t)^{2n-k+2}}f^{(r)}(t)dt, \,\, x>0.
\end{eqnarray*}
\end{lemma}
\begin{proof}
 Using the substitution $t=vx$ in (\ref{eq1}), we obtain
\begin{eqnarray*}
M_{n,k}(f;x)=\beta_{n}\int_{0}^{\infty}\frac{v^{n-k}}{(1+v)^{2n-k+2}}f(vx)dv.
\end{eqnarray*}
Using Leibniz's rule $r (r = 0, 1, 2...)$ times, we obtain
\begin{eqnarray*}
M_{n,k}^{(r)}(f;x)&=& \beta_{n}\frac{d^{r}}{dx^{r}}\int_{0}^{\infty}\frac{v^{n-k}}{(1+v)^{2n-k+2}}f(vx)dv\\
&=& \beta_{n}\int_{0}^{\infty}\frac{v^{n-k}}{(1+v)^{2n-k+2}}\frac{\partial^{r}}{\partial x^{r}}f(vx)dv\\
&=& \beta_{n}\int_{0}^{\infty}\frac{v^{n-k+r}}{(1+v)^{2n-k+2}}f^{(r)}(vx)dv.
\end{eqnarray*}
Using $v=\frac{t}{x}$, we get
\begin{eqnarray*}
M_{n,k}^{(r)}(f;x)=\beta_{n}x^{n+1-r}\int_{0}^{\infty}\frac{t^{n-k+r}}{(x+t)^{2n-k+2}}f^{(r)}(t)dt.
\end{eqnarray*}
\end{proof}
Next, we define
\begin{eqnarray*}
M_{n,k,r}^{\ast}(g;x)=\frac{\beta_{n}x^{n+1-r}}{b(n,k,r)}\int_{0}^{\infty}\frac{t^{n-k+r}}{(x+t)^{2n-k+2}}g(t)dt,
\end{eqnarray*}
where
\begin{eqnarray*}
b(n,k,r)&=&\beta_{n}x^{n+1-r}\int_{0}^{\infty}\frac{t^{n-k+r}}{(x+t)^{2n-k+2}}dt=\frac{(n-r)!(n-k+r)!}{n!(n-k)!}.
\end{eqnarray*}
Let us define
\begin{eqnarray*}
e_{m}(t)=t^{m}, \,\,\, \varphi_{x,m}(t)=(t-x)^{m},  \,\,\, m\in N_{0},\,\,\, x, t\in(0,\infty).
\end{eqnarray*}

\begin{lemma}\label{l3}
For any $m\in N_{0}$, $m\leq n-r$ and $ r\leq n$
\begin{eqnarray}\label{eq3}
M_{n,k,r}^{\ast}(e_{m};x)=\frac{(n-r-m)!(n-k+r+m)!}{(n-r)!(n-k+r)!}x^{m};
\end{eqnarray}
and
\begin{eqnarray}\label{eq4}
M_{n,k,r}^{\ast}(\varphi_{x,m};x)=\left(\sum_{j=0}^{m}(-1)^{j}{m\choose j}\frac{(n-r-m+j)!(n-k+r+m-j)!}{(n-r)!(n-k+r)!}\right)x^{m},
\end{eqnarray}
for each $x\in(0,\infty).$
\end{lemma}
\begin{proof}
The proof of (\ref{eq3}) is follows from \cite{AII}. On the other hand, we have the following identity,
\begin{eqnarray*}
(t-x)^{m}=\sum_{j=0}^{m}(-1)^{j}{m\choose j}x^{j}t^{m-j}.
\end{eqnarray*}
Then, we have
\begin{eqnarray*}
M_{n,k,r}^{\ast}((t-x)^{m};x)&=&\int_{0}^{\infty}K_{n,k}(x,t)(t-x)^{m}dt\\
&=&\int_{0}^{\infty}K_{n,k}(x,t)\sum_{j=0}^{m}(-1)^{j}{m\choose j}x^{j}t^{m-j}dt\\
&=&\sum_{j=0}^{m}(-1)^{j}{m\choose j}x^{j}M_{n,k,r}^{\ast}(t^{m-j};x).
\end{eqnarray*}
Using (\ref{eq3}), we get (\ref{eq4}).
\end{proof}

\begin{lemma}\label{l5} For $m=0,1,2,3,4$, one has
\begin{enumerate}[(i)]
   \item $ M_{n,k,r}^{\ast}(\varphi_{x,0};x)=\displaystyle1,$
   \item $ M_{n,k,r}^{\ast}(\varphi_{x,1};x)=\displaystyle\frac{2r-k+1}{n-r}x,$
   \item $ M_{n,k,r}^{\ast}(\varphi_{x,2};x)=\displaystyle\frac{4r^{2}+4r(2-k)+2n+k^{2}-5k+4}{(n-r)(n-r-1)}x^{2},$
   \item $ M_{n,k,r}^{\ast}(\varphi_{x,3};x)=\displaystyle\frac{c_{n,k,r}}{(n-r)(n-r-1)(n-r-2)}x^{3},$
   \item $ M_{n,k,r}^{\ast}(\varphi_{x,4};x)=\displaystyle\frac{d_{n,k,r}}{(n-r)(n-r-1)(n-r-2)(n-r-3)}x^{4},$
\end{enumerate}
where $c_{n,k,r}=8r^{3}+r^{2}(36-2k)+r(51+14n-42k+6k^{2})-k^{3}+12k^{2}-34k-n^{2}+n(17-6k-6k^{2}+2kr)+21$ and\\ $d_{n,k,r}=16r^{4}+r^{3}(128-32k)+r^{2}(348+48n-216k+24k^{2})+r(366+177n+k(6n^{2}-54n-440)+120k^{2}-8k^{3})+k^{4}+k^{3}(4n-22)+139k^{2}-k(245+116n)+24n^{2}+131n+100$.
\end{lemma}

\section{\textbf{Voronovskaja type theorem}}
In this section we obtain the Voronovskaja type theorem for the operators $M_{n,k}^{(r)}$.\\
Let $C_{B}[0,\infty)$ be the space of all real valued continuous and bounded functions on $[0,\infty)$ endowed with the usual supremum norm.
By $C_{B}^{(r+2)}[0,\infty)$($r\in N_{0}$), we denote the space of all functions $f\in C_{B}[0,\infty)$ such that $f', f'',...f^{(r+2)}\in C_{B}[0,\infty)$.
\begin{thm}{\label{t1}} Let $f$ be integrable in $(0,\infty)$ and admits its $(r+1)^{th}$ and $(r+2)^{th}$ derivatives, which are bounded at a fixed point $x\in(0,\infty)$ and $f^{(r)}(t)=O(t^{\alpha})$, as $t\rightarrow\infty$ for some $\alpha>0$, then
\begin{eqnarray*}
\lim_{n\rightarrow\infty}n\left(\frac{1}{b(n,k,r)}M_{n,k}^{(r)}(f;x)-f^{(r)}(x)\right)= (2r-k+1)x f^{(r+1)}(x)+x^{2}f^{(r+2)}(x)
\end{eqnarray*}
holds.
\end{thm}
\begin{proof}
Using Taylor's theorem, we get
\begin{eqnarray*}
f^{(r)}(t)-f^{(r)}(x)=(t-x)f^{(r+1)}(x)+\frac{1}{2}(t-x)^{2}f^{(r+2)}(x)+(t-x)^{2}\xi(t,x),
\end{eqnarray*}
where $\xi(t,x)$ is the peano form of the remainder and $\displaystyle\lim_{t\rightarrow x}\xi(t,x)=0$.\\
Then, we have\\
\noindent
$\frac{1}{b(n,k,r)}M_{n,k}^{(r)}(f;x)-f^{(r)}(x)$
\begin{eqnarray*}
&=&\frac{\beta_{n}}{b(n,k,r)}x^{n+1-r}\int_{0}^{\infty}\frac{t^{n-k+r}}{(x+t)^{2n-k+2}}\bigg(f^{(r)}(t)-f^{(r)}(x)\bigg)dt\\
&=&\frac{\beta_{n}}{b(n,k,r)}x^{n+1-r}\int_{0}^{\infty}\frac{t^{n-k+r}}{(x+t)^{2n-k+2}}\left((t-x)f^{(r+1)}(x)+\frac{1}{2}(t-x)^{2}f^{(r+2)}(x)+(t-x)^{2}\xi(t,x)\right)dt\\
&=& f^{(r+1)}(x)M_{n,k,r}^{\ast}(t-x,x)+\frac{1}{2}f^{(r+2)}(x)M_{n,k,r}^{\ast}((t-x)^{2},x)+M_{n,k,r}^{\ast}((t-x)^{2}\xi(t,x);x).
\end{eqnarray*}
Using Lemma \ref{l5}, we get
\begin{eqnarray*}
n\bigg(\frac{1}{b(n,k,r)}M_{n,k}^{(r)}(f;x)-f^{(r)}(x)\bigg)=\frac{n(2r-k+1)}{n-r}xf^{(r+1)}(x)+\frac{n(4r^{2}+4r(2-k)+2n+k^{2}-5k+4)}{2(n-r)(n-r-1)}\\
x^{2}f^{(r+2)}(x)+
n M_{n,k,r}^{\ast}((t-x)^{2}\xi(t,x);x).
\end{eqnarray*}
By using Cauchy-Schwarz inequality, we have
\begin{eqnarray}\label{eq6}
n\left(M_{n,k,r}^{\ast}((t-x)^{2}\xi(t,x);x)\right)\leq\sqrt{n^{2}M_{n,k,r}^{\ast}(\varphi_{x,4};x)}\sqrt{M_{n,k,r}^{\ast}(\xi^{2}(t,x);x)}.
\end{eqnarray}
We observe that $\xi^{2}(x,x)=0$ and $\xi^{2}(.,x)$ is continuous at $t\in(0,\infty)$ and bounded as $t\rightarrow\infty$. Then from Korovkin theorem that
\begin{eqnarray}\label{eq7}
\lim_{n\rightarrow\infty}M_{n,k,r}^{\ast}(\xi^{2}(t,x);x)=\xi^{2}(x,x)=0,
\end{eqnarray}
in view of fact that $M_{n,k,r}^{\ast}(\varphi_{x,4};x)=O\bigg(\displaystyle\frac{1}{n^{2}}\bigg).$ Now, from (\ref{eq6}) and (\ref{eq7}) we obtain
\begin{eqnarray}\label{eq8}
\lim_{n\rightarrow\infty}n M_{n,k,r}^{\ast}\left((t-x)^{2}\xi(t,x);x\right)=0.
\end{eqnarray}
Using (\ref{eq8}), we have
\begin{eqnarray*}
\lim_{n\rightarrow\infty}n\left(\frac{1}{b(n,k,r)}M_{n,k}^{(r)}(f;x)-f^{(r)}(x)\right)= (2r-k+1)x f^{(r+1)}(x)+x^{2}f^{(r+2)}(x).
\end{eqnarray*}
This completes the proof.
\end{proof}

\section{\textbf{Direct results}}
In this section we obtain the rate of convergence of the operators $M_{n,k}^{(r)}$.\\
Let us consider the following K-functional:
\begin{eqnarray}\label{e11}
  K(f,\delta) &=& \inf_{g\in C_{B}^{2}[0,\infty)}\{\parallel f-g \parallel+\delta \parallel g''\parallel\},
\end{eqnarray}
where $\delta>0$. By, p. 177, Theorem 2.4 in \cite{RAD}, there exists an absolute constant $C>0$ such that
\begin{eqnarray}\label{e12}
K(f,\delta) \leq C\omega_{2}(f,\sqrt\delta),
\end{eqnarray}
where
\begin{eqnarray}\label{e13}
 \omega_{2}(f,\sqrt\delta)  &=& \sup_{0<h\leq\sqrt\delta} \sup_{x \in [0,\infty)}\mid f(x+2h)-2f(x+h)+f(x) \mid
\end{eqnarray}
is the second order modulus of smoothness of $f$. By
\begin{eqnarray*}
  \omega(f,\delta) &=&  \sup_{0<h\leq\delta} \sup_{x \in [0,\infty)}\mid f(x+h)-f(x) \mid ,
\end{eqnarray*}
we denote the first order modulus of continuity of $f$ and satisfies the following property:
\begin{eqnarray}\label{e14}
|f(t)-f(x)|\leq\bigg(1+\frac{|t-x|}{\delta}\bigg)\omega(f,\delta),
\end{eqnarray}
where $\delta>0$.\\

\begin{thm}{\label{t2}}
Let $f\in C_{B}^{r}[0,\infty)$ and $r\in N_{0}$. Then for $n>r$, we have
\begin{eqnarray*}
\bigg|\frac{1}{b(n,k,r)}M_{n,k}^{(r)}(f;x)-f^{(r)}(x)\bigg|\leq 2\omega(f^{(r)},\sqrt{\delta_{n}}),
\end{eqnarray*}
where
\begin{eqnarray*}
\delta_{n}=\left(\frac{4r^{2}+4r(2-k)+2n+k^{2}-5k+4}{(n-r)(n-r-1)}\right)x^{2}.
\end{eqnarray*}
\end{thm}
\begin{proof}
By using monotonicity of $ M_{n,k,r}^{\ast}$, we get\\
\noindent
$\left|\frac{1}{b(n,k,r)}M_{n,k}^{(r)}(f;x)-f^{(r)}(x)\right|$
\begin{eqnarray*}
&=&\left|\frac{\beta_{n}}{b(n,k,r)}x^{n+1-r}\int_{0}^{\infty}\frac{t^{n-k+r}}{(x+t)^{2n-k+2}}(f^{(r)}(t)-f^{(r)}(x))dt\right|\\
&=&|M_{n,k,r}^{\ast}((f^{(r)}(t)-f^{(r)}(x));x)|\\
&\leq& M_{n,k,r}^{\ast}(|f^{(r)}(t)-f^{(r)}(x)|;x)\\
&\leq&\omega\left(f^{(r)},\delta\right)\frac{\beta_{n}}{b(n,k,r)}x^{n+1-r}\int_{0}^{\infty}\frac{t^{n-k+r}}{(x+t)^{2n-k+2}}\bigg(1+\frac{|t-x|}{\delta}\bigg)dt\\
&\leq&\omega\left(f^{(r)},\delta\right)\bigg(1+\frac{1}{\delta}\frac{\beta_{n}}{b(n,k,r)}x^{n+1-r}\int_{0}^{\infty}\frac{t^{n-k+r}}{(x+t)^{2n-k+2}}|t-x|dt\bigg).
\end{eqnarray*}
Thus, by applying the Cauchy-Schwarz inequality, we have
\begin{eqnarray*}
\left|\frac{1}{b(n,k,r)}M_{n,k}^{(r)}(f;x)-f^{(r)}(x)\right|\leq\omega\left(f^{(r)},\delta\right)\bigg(1+\frac{1}{\delta}\left(M_{n,k,r}^{\ast}((t-x)^{2};x)\right)^{1/2}\bigg).
\end{eqnarray*}
Choosing $ \delta=\sqrt{\delta_{n}}$, we have
\begin{eqnarray*}
\bigg|\frac{1}{b(n,k,r)}M_{n,k}^{(r)}(f;x)-f^{(r)}(x)\bigg|\leq 2\omega\left(f^{(r)},\sqrt{\delta_{n}}\right).
\end{eqnarray*}
Hence, the proof is completed.
\end{proof}

\begin{thm}{\label{t3}}
Let $f\in C_{B}^{r}[0,\infty)$ and $r\in N_{0}$. Then for $n>r$, we have
\begin{eqnarray*}
\bigg|\frac{1}{b(n,k,r)}M_{n,k}^{(r)}(f;x)-f^{(r)}(x)\bigg|\leq C\omega_{2}\left(f^{(r)},\gamma_{n}\right)+\omega\left(f^{(r)},\frac{2r-k+1}{n-r}x\right),
\end{eqnarray*}
where $C$ is an absolute constant and
\begin{eqnarray*}
\gamma_{n}=\left(\frac{4r^{2}+4r(2-k)+2n+k^{2}-5k+4}{(n-r)(n-r-1)}x^{2}+\bigg(\frac{2r-k+1}{n-r}x\bigg)^{2}\right)^{1/2}.
\end{eqnarray*}
\end{thm}
\begin{proof}
Let us consider the auxiliary operators $\overline{M_{n,k,r}^{\ast}}$ defined by
\begin{eqnarray}\label{eq9}
\overline{M_{n,k,r}^{\ast}}(f;x)=M_{n,k,r}^{\ast}(f;x)-f\left(x+\frac{2r-k+1}{n-r}x\right)+f(x).
\end{eqnarray}
Using Lemma \ref{l5}, we observe that the operators $\overline{M_{n,k,r}^{\ast}}$ are linear and reproduce the linear functions.\\
Hence
\begin{eqnarray}\label{eq10}
\overline{M_{n,k,r}^{\ast}}((t-x);x)=0.
\end{eqnarray}
Let $g\in C_{B}^{r+2}[0,\infty)$ and $x\in(0,\infty)$. By Taylor's theorem, we have
\begin{eqnarray*}
g^{(r)}(t)-g^{(r)}(x)=(t-x)g^{(r+1)}(x)+\int_{x}^{t}(t-v)g^{(r+2)}(v)dv,  \,\,\, t\in(0,\infty).
\end{eqnarray*}
Using (\ref{eq9}) and (\ref{eq10}), we get\\
\noindent
$|\overline{M_{n,k,r}^{\ast}}(g^{(r)};x)-g^{(r)}(x)|$
\begin{eqnarray*}
&=& \left|g^{(r+1)}(x)\overline{M_{n,k,r}^{\ast}}(t-x;x)+\overline{M_{n,k,r}^{\ast}}\left(\int_{x}^{t}(t-v)g^{(r+2)}(v)dv;x\right)\right|\\
&\leq&
\left|M_{n,k,r}^{\ast}\left(\int_{x}^{t}(t-v)g^{(r+2)}(v)dv;x\right)\right|+\left|\int_{x}^{x+\frac{2r-k+1}{n-r}x}\left(x+\frac{2r-k+1}{n-r}x-v\right)g^{(r+2)}(v)dv\right|.
\end{eqnarray*}
Observe that
\begin{eqnarray*}
\left|M_{n,k,r}^{\ast}\left(\int_{x}^{t}(t-v)g^{(r+2)}(v)dv;x\right)\right|\leq||g^{(r+2)}||M_{n,k,r}^{\ast}((t-x)^{2};x)
\end{eqnarray*}
and
\begin{eqnarray*}
\left|\int_{x}^{x+\frac{2r-k+1}{n-r}x}\left(x+\frac{2r-k+1}{n-r}x-v\right)g^{(r+2)}(v)dv\right|\leq||g^{(r+2)}||\left(\frac{2r-k+1}{n-r}x\right)^{2}.
\end{eqnarray*}
Hence by Lemma \ref{l5}, we have
\begin{eqnarray}\label{eq11}
|\overline{M_{n,k,r}^{\ast}}(g^{(r)};x)-g^{(r)}(x)|\leq||g^{(r+2)}||\left(\frac{4r^{2}+4r(2-k)+2n+k^{2}-5k+4}{(n-r)(n-r-1)}x^{2}+\left(\frac{2r-k+1}{n-r}x\right)^{2}\right).
\end{eqnarray}
Now $g\in C_{B}^{r+2}[0,\infty)$, using (\ref{eq11}), we obtain
\begin{eqnarray*}
\left|\frac{1}{b(n,k,r)}M_{n,k}^{(r)}(f;x)-f^{(r)}(x)\right|&=&|M_{n,k,r}^{\ast}(f^{(r)};x)-f^{(r)}(x) |\\
&\leq&|\overline{M_{n,k,r}^{\ast}}(f^{(r)}-g^{(r)};x)-(f^{(r)}-g^{(r)})(x)|+|\overline{M_{n,k,r}^{\ast}}(g^{(r)};x)-g^{(r)}(x)|\\
&+&\left|f^{(r)}\left(x+\frac{2r-k+1}{n-r}x\right)-f^{(r)}(x)\right|\\
&\leq&4||f^{(r)}-g^{(r)}||+\gamma_{n}^{2}||g^{(r+2)}||+\omega\left(f^{(r)},\frac{2r-k+1}{n-r}x\right).
\end{eqnarray*}
Taking infimum over all $g\in C_{B}^{r+2}[0,\infty)$, we obtain
\begin{eqnarray*}
\left|\frac{1}{b(n,k,r)}M_{n,k}^{(r)}(f;x)-f^{(r)}(x)\right|\leq K\left(f^{(r)},\gamma_{n}^{2}\right)+\omega\left(f^{(r)},\frac{2r-k+1}{n-r}x\right).
\end{eqnarray*}
Using (\ref{e12}), we have
\begin{eqnarray*}
\bigg|\frac{1}{b(n,k,r)}M_{n,k}^{(r)}(f;x)-f^{(r)}(x)\bigg|\leq C\omega_{2}\left(f^{(r)},\gamma_{n}\right)+\omega\left(f^{(r)},\frac{2r-k+1}{n-r}x\right).
\end{eqnarray*}
Hence, the proof is completed.
\end{proof}

{\bf Acknowledgements}
The author is extremely grateful to Prof. P. N. Agrawal, IITR, India for making valuable suggestions leading to the overall improvements in the paper.

\end{document}